\documentclass{amsart}
\usepackage{amsmath}
\usepackage{amssymb}
\usepackage{color}
\usepackage{mathdots}

\usepackage{graphicx}
\usepackage{epsfig}
\usepackage{multirow}
\usepackage{colortbl}
\usepackage[table]{xcolor}

\definecolor{lightgray}{gray}{0.9}

\DeclareMathOperator{\Id}{Id}

\def\FF{\mathbb{F}}

\newtheorem{theorem}{Theorem}[section]
\newtheorem{lemma}[theorem]{Lemma}
\newtheorem{proposition}[theorem]{Proposition}
\newtheorem{corollary}[theorem]{Corollary}
\theoremstyle{definition}

\theoremstyle{remark}
\newtheorem{remark}[theorem]{Remark}
\numberwithin{equation}{section}

\def\separa{\hbox to 14 truecm{\hrulefill}}

\author[P. Danchev]{Peter Danchev}
\address{Institute of Mathematics and Informatics, Bulgarian Academy of Sciences, 1113 Sofia, Bulgaria}
\email{danchev@math.bas.bg}
\thanks{The first author was partially supported  by the BIDEB 2221 of T\"UB\'ITAK}

\author[E. Garc\'\i a]{Esther Garc\'\i a}
\address{Departamento de Matem\'{a}tica  Aplicada, Ciencia e Ingenier\'{\i}a de los Materiales y Tecnolog\'{\i}a Electr\'onica,
Universidad Rey Juan Carlos, 28933 M\'{o}s\-to\-les (Madrid), Spain}
\thanks{The second author and third authors were partially supported by Ayuda Puente 2023, URJC, and MTM2017-84194-P (AEI/FEDER, UE)}
\email{esther.garcia@urjc.es}

\author[M. G\'omez Lozano]{Miguel G\'omez Lozano}
\address{Departamento de \'Algebra, Geometr\'{\i}a y
Topolog\'{\i}a, Universidad de M\'alaga, 29071 M\'alaga, Spain}
\thanks{The three authors were partially supported by  the Junta de Andaluc\'{\i}a FQM264}
\email{miggl@uma.es}

\begin{document}

\title[Prescribed characteristic polynomials]{On prescribed characteristic polynomials}
\maketitle

\begin{abstract} Let $\mathbb{F}$ be a field. We show that given any $n$th degree monic polynomial $q(x)\in \mathbb{F}[x]$  and any  matrix $A\in\mathbb{M}_n(\mathbb{F})$ whose trace coincides with the trace of $q(x)$ and consisting in its main diagonal of $k$ 0-blocks of order one, with $k<n-k$, and an invertible non-derogatory block of order $n-k$, we can construct a square-zero matrix $N$ such that the characteristic polynomial of $A+N$ is exactly $q(x)$.
We also show that the restriction $k<n-k$ is necessary in the sense that, when the equality $k=n-k$ holds, not every characteristic polynomial having the same trace as $A$ can be obtained by adding a square-zero matrix.
Finally, we apply our main result to decompose matrices into the sum of a square-zero matrix and some other matrix which is either  diagonalizable, invertible, potent or torsion.
\end{abstract}

\bigskip
{\footnotesize \textit{Key words}: characteristic polynomial, square-zero matrix}

{\footnotesize \textit{2010 Mathematics Subject Classification}: 15A15, 15A21,  15A83}

\section{Introduction and Basic Facts}

Throughout the current article, we denote by $\mathbb{M}_n(\mathbb{F})$ the matrix ring consisting of all square matrices of size $n\times n$ over an arbitrary field $\mathbb{F}$.

The problem of completing a matrix, when some of its entries are fixed, such that the resulting matrix satisfies a prescribed characteristic polynomial is classical and has been studied by several authors; see, for more detailed information, the survey by G. Cravo \cite{C}. Back in 1958, H.K. Farahat and W. Ledermann showed that any matrix of the form
$$
A=\left(
  \begin{array}{c|ccc}
    0 & 0 & \dots & 0\\
    \hline
    0 &  &  & \\
    \vdots &  & C & \\
    0 &  &  & \\
  \end{array}
\right)\in  \mathbb{M}_{n}(\mathbb{F}),
$$
where $C\in \mathbb{M}_{n-1}(\mathbb{F})$ is non-derogatory, can be modified by adding a matrix of the form
$$
N=\left(
  \begin{array}{c|ccc}
    a & u_1 & \dots & u_{n-1}\\
    \hline
    v_1 &  &  & \\
    \vdots &  & {\bf 0}_{n-1,n-1} & \\
    v_{n-1} &  &  & \\
  \end{array}
\right)\in  \mathbb{M}_{n}(\mathbb{F})
$$
such that the characteristic polynomial of $A+N$ is any prescribed monic polynomial of degree $n$ (cf. \cite[Theorem 3.4]{FL}). Furthermore, this problem was generalized by G.N. Oliveira who raised the question of when a block matrix of the form
$$
\left(
  \begin{array}{c|c}
    A_{11} & A_{12} \\
    \hline
    A_{21} & A_{22} \\
  \end{array}
\right),
$$
where part of its blocks are known, can be completed to achieve any prescribed characteristic polynomial, and addressed it in several of his works (see \cite{O71}, \cite{O75}, \cite{O81}, \cite{O82}).

In particular, in his work \cite{O71}, Oliveira proved that any matrix of the form
$$
A=\left(
  \begin{array}{c|c }
     {\bf 0}_{k,k} &  {\bf 0}_{k,n-k} \\
    \hline
     {\bf 0}_{n-k,k} & A_{22}  \\
  \end{array}
\right)\in  \mathbb{M}_{n}(\mathbb{F}),
$$
where $A_{22}\in \mathbb{M}_{n-k}(\mathbb{F})$ is fixed, can be completed by adding a matrix of the form
$$
N=\left(
  \begin{array}{c|c  }
    X_{11} & X_{12}\\
    \hline
   X_{21} &  {\bf 0}_{n-k,n-k}\\
    \end{array}
\right)\in  \mathbb{M}_{n}(\mathbb{F})
$$
such that the characteristic polynomial of $A+N$ is any prescribed monic polynomial $q(x)$ of degree $n$ if, and only if, the product of certain elementary divisors of the matrix $A_{22}$ divides $q(x)$.

The prescription of several entries of a matrix together with a fixed characteristic polynomial has been recently used by S. Breaz and G. Călugăreanu when showing that every square matrix with nilpotent trace over a general ring can be expressed as the sum of three nilpotent matrices; see \cite{BrC}.

With the aim of addressing decomposition problems involving square-zero matrices, in this work we study the problem of modifying a matrix of the form
$$
A=\left(
  \begin{array}{c|c }
     {\bf 0}_{k,k} &  {\bf 0}_{k,n-k} \\
    \hline
     {\bf 0}_{n-k,k} & A_{22}  \\
  \end{array}
\right)\in  \mathbb{M}_{n}(\mathbb{F}),
$$
with fixed $A_{22}\in \mathbb{M}_{n-k}(\mathbb{F})$, by adding to it a square-zero matrix $N$ such that the characteristic polynomial of $A+N$ is any prescribed polynomial. Since square-zero matrices always have zero trace, the trace of $A+N$ equals the trace of $A$, and by this modification of $A$ we can only get characteristic polynomials having the same trace as $A$. We will show in the sequel that, when $k<n-k$ and $A_{22}$ is invertible and non-derogatory, given any monic degree $n$ polynomial $q(x)$ with the same trace as $A$ there exists a square-zero matrix $N$ such that the characteristic polynomial of $A+N$ is precisely $q(x)$. The first and last $k$ rows of such matrix $N$ will be non-zero in general, so this is not exactly a completion problem, because, in order to get the desired characteristic polynomial by adding a square-zero matrix to $A$, we modify the last $k$ rows of $A_{22}$.

We also illustrate that the restriction $k<n-k$ cannot be generally ignored in the sense that, when the equality $k=n-k$ holds, not every characteristic polynomial having the same trace as $A$ can be obtained by adding a square-zero matrix.

Further, in the last section of this work, we will apply this result to the problem of decomposing matrices into the sum of a square-zero matrix and a matrix satisfying some ``good" condition, where ``good" stands for the properties of being either diagonalizable, invertible, $n$-torsion or $n$-potent. This fits in our general project of decomposing square matrices into the sum of a square-zero matrix and another matrix satisfying some  properties such as been diagonalizable, potent, invertible, etc. (see, for instance, \cite{DGL1}, \cite{DGL2}, \cite{DGL3} and \cite{DGL4}).

\section{Main Theorem}

Recall that the {\it trace} of a monic polynomial $$q(x)=x^n+a_{n-1}x^{n-1}+\dots+a_0\in \mathbb{F}[x]$$ is $-a_{n-1}\in\mathbb{F}$. If such polynomial is the characteristic polynomial of a certain matrix, the trace coincides with the trace (i.e., the sum of the diagonal elements) of such matrix. Recall that a matrix $N$ is {\it nilpotent} if there exists $k\in \mathbb{N}$ with $N^k=0$, and it is called a {\it square-zero} matrix when $N^2=0$.

\begin{remark}\label{k0}
Given a non-derogatory matrix $A$, there always exists a square-zero matrix $N$ such that the characteristic polynomial of $A+N$ is any given monic polynomial whose trace coincides with the trace of $A$. In fact, we can suppose that $A$ is the companion matrix of a certain polynomial $$p(x)=x^n+u_{n-1}x^{n-1}+u_{n-2}x^{n-2}+\dots+u_0$$ thus:
$$
A=\left(
    \begin{array}{ccc c}
      0 & \dots &  0& -u_0 \\
      1 &  & 0 & \vdots \\
       & \ddots &  & -u_{n-2} \\
      0 &  & 1 & -u_{n-1} \\
    \end{array}
  \right).
$$ Given any polynomial $$q(x)=x^n+u_{n-1}x^{n-1}+v_{n-2}x^{n-2}+\dots+v_0$$ (where the trace of $q(x)$ coincides with the trace of $A$) it is enough to consider the square-zero matrix
$$
N=\left(
    \begin{array}{ccc c}
      0 &  &  0& u_0-v_0 \\
      0 &  & 0 & \vdots \\
      \vdots &   & \vdots & u_{n-2}-v_{n-2} \\
      0 &  &  0 & 0 \\
    \end{array}
  \right)
$$
to get that
$$
A+N=\left(
    \begin{array}{ccc c}
      0 &  & 0 & -v_0 \\
      1 &  & 0 & \vdots \\
       &  \ddots &  & -v_{n-2} \\
      0  &  &   1& -u_{n-1} \\
    \end{array}
  \right)
$$
has characteristic polynomial equal to $q(x)$.
\end{remark}

The goal of this section is to generalize the previous result to any invertible non-derogative matrix completed with $k$ rows/columns of zeros.
In order to do that, we begin with a  technical lemma.

\begin{lemma}\label{lemma2}
Let $\mathbb{F}$ be a field and let $a_1,\dots,a_n,u_2,\dots,u_{n}\in \FF$. Then the determinant of the matrix
$$A=
   \left(\begin{array}{ccccccc}
    a_1+x & a_2&a_3&\cdots&a_{n-2}&a_{n-1}&a_{n}  \\
     0 & x&0&\cdots&0&0& u_2 \\
     0 & 0&x&\cdots&0&0& u_3 \\
     0 & 0&0&\ddots&0&0& u_4 \\
     \vdots&\vdots&\vdots&\ddots&\ddots&\vdots&\vdots\\
     0 & 0&0&\cdots&0&x& u_{n-1} \\
     x & 0&0&\cdots&0&0&x+u_{n}  \\
\end{array}
\right)
$$
is $|A|=x^{n}+ u_{n}x^{n-1}+(a_1-a_{n})x^{n-1}+a_1u_{n}x^{n-2}+\sum_{i=2}^{n-1} a_iu_i x^{n-2}$.
\end{lemma}

\begin{proof}
By $R(i)$ and $C(i)$, we respectively denote the $i^{\rm th}$-row and the $i^{\rm th}$-column of the involved matrix. To perform the calculations, we will assume that our matrix is a matrix with coefficients in the field of fractions of  $\mathbb{F}[x]$.

For $s=2,\dots, n-1$, we add $-\frac{a_{s}}{x}R(s)$ to row $R(1)$ as follows:
$$|A|=
   \left|\begin{array}{ccccccc}
    a_1+x & 0&0&\cdots&0&0&a_{n}-\sum_{s=2}^{n-1} \frac{a_{s}u_s}{x} \\
     0 & x&0&\cdots&0&0& u_2 \\
     0 & 0&x&\cdots&0&0& u_3 \\
     0 & 0&0&\ddots&0&0& u_4 \\
     \vdots&\vdots&\vdots&\ddots&\ddots&\vdots&\vdots\\
     0 & 0&0&\cdots&0&x& u_{n-1} \\
     x & 0&0&\cdots&0&0&x+u_{n}  \\
  \end{array}
\right|.
$$
By the expansion of the determinant using $C(2)$, \dots, $C(n-2)$, we get:
\begin{align*}
    |A|&=x^{n-2}\left|
             \begin{array}{cc}
               a_1+x  & a_{n}-\sum_{s=2}^{n-1} \frac{a_{s}u_s}{x}  \\
               x & x+u_{n} \\
             \end{array}
           \right|\\&=x^{n-2}(a_1+x)(x+u_{n})-x^{n-1}(a_{n}-\sum_{s=2}^{n-1} \frac{a_{s}u_s}{x})\\
           &=x^{n}+ u_{n}x^{n-1}+(a_1-a_{n})x^{n-1}+a_1u_{n}x^{n-2}+\sum_{i=2}^{n-1} a_iu_i x^{n-2}.
\end{align*}
\end{proof}

We are now prepared to establish the following chief result.

\begin{theorem}\label{teoremappal}
Let $\mathbb{F}$ be a field, let $n,k\in \mathbb{N}$ with $k<n-k$, and consider the block matrix
$$
A=\left(
         \begin{array}{c|c}
           {\bf 0}_{k,k} & {\bf 0}_{k,n-k} \\
           \hline
           {\bf 0}_{n-k,k} & A_{22} \\
         \end{array}
       \right)\in \mathbb{M}_{n}(\mathbb{F})
$$
consisting of $k$ rows and columns of zeros and an invertible non-derogatory matrix $A_{22}$.
Then, for any monic polynomial $q(x)$ of degree $n$ whose trace coincides with the trace of $A$, there exists a square-zero matrix $N$
such that the characteristic polynomial of $A+N$ coincides with $q(x)$.
\end{theorem}

\begin{proof} The case $k=0$ was treated in Remark \ref{k0}, so we can suppose that $k>0$.
Under an appropriate change of the existing basis, we can also suppose that $A_{22}$ is a companion matrix $C(p(x))$ for a polynomial of the form $p(x)=x^{n-k}+u_{n-k-1}x^{n-k-1}+\dots+u_0$.

Take arbitrary $a_{ij}\in \FF$, $i=1,\dots, k$, $j=0,\dots, n-2k$, and define the matrix
$N:=(\alpha_{ij})\in \mathbb{M}_{n}(\mathbb{F})$ like this:
$$\alpha_{i,j}:=\begin{cases}
     -a_{i,n-k+1}&\ \text{ if }\ i=1,\dots,k \text{ and } j=k,\\
     -a_{i,j}&\ \text{ if }\  i=1,\dots,k \text{ and }j=k+1,\dots, n-k+1,\\
      a_{n-i+1,n-k+1}&\ \text{ if }\  i=n-k+1,\dots,n \text{ and } j=k,\\
     a_{n-i+1,j}&\ \text{ if }\  i=n-k+1,\dots,n \text{ and } j=k+1,\dots, n-k+1,\\
     -1&\ \text{ if }\  i=1,\dots,k-1 \text{ and } j=i, n-i+1,\\
     1&\ \text{ if }\  i=n-k+1,\dots,n \text{ and }  j=i, n-i+1,\\
     0&\ \text{ otherwise. }
\end{cases} $$
Let us first show that $N^2=0$: in fact, if we denote by $\{e_1,\dots, e_n\}$ the vectors of the canonical basis of $\mathbb{F}^n$, and by $n:\mathbb{F}^n\to \mathbb{F}^n$ the endomorphism whose matrix with respect to the canonical basis is $N$, one can check that
$$
n(e_j)\in {\rm span}\{e_1-e_n, e_2-e_{n-1}, \dots, e_k-e_{n-k+1}\},\qquad j=1,\dots, n
$$
and $n(e_i-e_{n-i+1})=0$, $i=1,\dots, k$, so $n^2=0$, i.e., $N$ is really a square-zero matrix, as claimed.

Let us calculate the characteristic polynomial of the matrix $A+N$, i.e., let us calculate the determinant of the matrix $M=x\Id_n-(A+N)$. We will perform  several elementary transformations on the rows and columns of $M$ in order to obtain a  matrix of the form of Lemma \ref{lemma2}, and then we will use the formula for the determinant of such matrix proved in Lemma \ref{lemma2}. To this purpose, we denote by $R(i)$ and $C(i)$ the respective $i^{\rm th}$-row and the $i^{\rm th}$-column of the involved matrices. For our calculations, we will assume that our matrix is a matrix with coefficients in the field of fractions of the polynomial ring possessing coefficients in $\mathbb{F}$.

Since it is not possible to explicitly write the generic matrix $M$, we will represent the transformations on the example $n=10$ and $k=4$. In this concrete case, we have:
{\small $$
A=\left(
      \begin{array}{cccc|cccccc}
         0 & 0 & 0 & 0 & 0 & 0 & 0 & 0 & 0 &   0\\
          0 & 0 & 0 & 0 & 0 & 0 & 0 & 0 & 0 & 0 \\
          0 & 0 & 0 & 0 & 0 & 0 & 0 & 0 & 0 & 0 \\
         0 & 0 & 0 & 0 & 0 & 0 & 0 & 0 & 0 & 0 \\
        \hline
        0 & 0 & 0 & 0 & 0 & 0 & 0 & 0 & 0 & -u_0 \\
        0 & 0 & 0 & 0 & 1&  0 & 0 & 0 & 0 & -u_1 \\
        0 & 0 & 0 & 0 & 0 & 1 & 0 & 0 & 0 & -u_2 \\
       0 & 0 & 0  &  0 &  0 &  0 & 1 & 0& 0 & -u_3 \\
       0 & 0 & 0 &  0 &  0 &  0 &  0 & 1 &0& -u_4 \\
        0 & 0 & 0 & 0  & 0  &  0 & 0  & 0 & 1 & -u_5 \\
      \end{array}
    \right),
$$
$$
N=\left(
      \begin{array}{cccccccccc}
         -1 & 0 & 0 & -a_{1,7} & -a_{1,5} & -a_{1,6} &-a_{1,7} & 0 & 0 &   -1\\
          0 &-1& 0 & -a_{2,7} & -a_{2,5} & -a_{2,6} & -a_{2,7} & 0 & -1 & 0 \\
          0 & 0 & -1 & -a_{3,7} & -a_{3,5} & -a_{3,6} & -a_{3,7}& -1 & 0 & 0 \\
          0 & 0 & 0 & -a_{4,7} & -a_{4,5} & -a_{4,6} & -a_{4,7}& 0 & 0 & 0 \\
            0 & 0 & 0 & 0 & 0 & 0 & 0 & 0 & 0 & 0 \\
        0 & 0 & 0 & 0 & 0&  0 & 0 & 0 & 0 &0 \\
         0 & 0 & 0 & a_{4,7} & a_{4,5} & a_{4,6} & a_{4,7}& 0 & 0 & 0 \\
    0 & 0 & 1 & a_{3,7} & a_{3,5} & a_{3,6} & a_{3,7}& 1 & 0 & 0 \\
      0 & 1& 0 & a_{2,7} & a_{2,5} & a_{2,6} & a_{2,7} & 0 & 1 & 0 \\
       1 & 0 & 0 & a_{1,7} & a_{1,5} & a_{1,6} & a_{1,7} & 0 & 0 &   1\\
      \end{array}
    \right),\hbox{ and}
$$
}
\scriptsize
$$M=\left(
      \begin{array}{cccccccccc}
        x+1 & 0 & 0 & a_{1,7} & a_{1,5} & a_{1,6} & a_{1,7} & 0 & 0 & 1 \\
        0 & x+1 & 0 & a_{2,7} & a_{2,5} & a_{2,6} & a_{2,7} & 0 & 1 & 0 \\
        0 & 0 & x+1 & a_{3,7} & a_{3,5} & a_{3,6} & a_{3,7} & 1 & 0 & 0 \\
        0 & 0 & 0 & x+a_{4,7} & a_{4,5} & a_{4,6} & a_{4,7} & 0 & 0 & 0 \\
        0 & 0 & 0 & 0 & x & 0 & 0 & 0 & 0 & u_0 \\
        0 & 0 & 0 & 0 & -1 & x & 0 & 0 & 0 & u_1 \\
        0 & 0 & 0 & -a_{4,7} & -a_{4,5} & -a_{4,6} & x-a_{4,7} & 0 & 0 & u_2 \\
       0 & 0 & -1 & -a_{3,7} & -a_{3,5} & -a_{3,6} & -a_{3,7}-1 & x-1 & 0 & u_3 \\
       0 & -1 & 0 & -a_{2,7} & -a_{2,5} & -a_{2,6} & -a_{2,7} & -1 & x-1 & u_4 \\
       -1 & 0 & 0 & -a_{1,7} & -a_{1,5} & -a_{1,6} & -a_{1,7} & 0 & -1 & x-1+u_5 \\
      \end{array}
    \right).$$\normalsize

\noindent (1) For $s=1,2,.... k$, we  add row $R(s)$ to row $R(n-s+1)$:
\scriptsize
$$\left(
      \begin{array}{cccccccccc}
        x+1 & 0 & 0 & a_{1,7} & a_{1,5} & a_{1,6} & a_{1,7} & 0 & 0 & 1 \\
        0 & x+1 & 0 & a_{2,7} & a_{2,5} & a_{2,6} & a_{2,7} & 0 & 1 & 0 \\
        0 & 0 & x+1 & a_{3,7} & a_{3,5} & a_{3,6} & a_{3,7} & 1 & 0 & 0 \\
        0 & 0 & 0 & x+a_{4,7} & a_{4,5} & a_{4,6} & a_{4,7} & 0 & 0 & 0 \\
        0 & 0 & 0 & 0 & x & 0 & 0 & 0 & 0 & u_0 \\
        0 & 0 & 0 & 0 & -1 & x & 0 & 0 & 0 & u_1 \\
        0 & 0 & 0 & x & 0 & -1 & x & 0 & 0 & u_2 \\
       0 & 0 & x & 0 &0  & 0 & -1 & x & 0 & u_3 \\
       0 & x & 0 & 0 & 0 & 0 & 0 & -1 & x & u_4 \\
       x & 0 & 0 & 0 & 0 & 0 & 0 & 0 & -1 & x+u_5 \\
      \end{array}
    \right)$$\normalsize

\noindent (2) For $s=2,.... k$, we add  $-\sum_{i=2}^{s} \frac{1}{x^{i-2}}C(n-s+i-1)$ to column $C(s)$:
\scriptsize
$$\left(
      \begin{array}{cccccccccc}
        x+1 & 0 & 0 & 0 & a_{1,5} & a_{1,6} & a_{1,7} & 0 & 0 & 1 \\
        0 & x & -1/x & -1/x^2 & a_{2,5} & a_{2,6} & a_{2,7} & 0 & 1 & 0 \\
        0 & 0 & x & -1/x & a_{3,5} & a_{3,6} & a_{3,7} & 1 & 0 & 0 \\
        0 & 0 & 0 & x & a_{4,5} & a_{4,6} & a_{4,7} & 0 & 0 & 0 \\
        0 & 0 & 0 & 0 & x & 0 & 0 & 0 & 0 & u_0 \\
        0 & 0 & 0 & 0 & -1 & x & 0 & 0 & 0 & u_1 \\
        0 & 0 & 0 & 0 & 0 & -1 & x & 0 & 0 & u_2 \\
       0 & 0 & 0 & 0 &0  & 0 & -1 & x & 0 & u_3 \\
       0 & 0 & 0 & 0 & 0 & 0 & 0 & -1 & x & u_4 \\
       x & 1 & 1/x & 1/x^2 & 0 & 0 & 0 & 0 & -1 & x+u_5 \\
      \end{array}
    \right)$$\normalsize

\noindent (3) For $s=2,.... k$, we  add  $-\frac{1}{x^{s-1}}C(1)$ to column $C(s)$:
\scriptsize
$$\left(
      \begin{array}{cccccccccc}
        x+1 & -(x+1)/x & -(x+1)/x^2 & -(x+1)/x^3 & a_{1,5} & a_{1,6} & a_{1,7} & 0 & 0 & 1 \\
        0 & x & -1/x & -1/x^2 & a_{2,5} & a_{2,6} & a_{2,7} & 0 & 1 & 0 \\
        0 & 0 & x & -1/x & a_{3,5} & a_{3,6} & a_{3,7} & 1 & 0 & 0 \\
        0 & 0 & 0 & x & a_{4,5} & a_{4,6} & a_{4,7} & 0 & 0 & 0 \\
        0 & 0 & 0 & 0 & x & 0 & 0 & 0 & 0 & u_0 \\
        0 & 0 & 0 & 0 & -1 & x & 0 & 0 & 0 & u_1 \\
        0 & 0 & 0 & 0 & 0 & -1 & x & 0 & 0 & u_2 \\
       0 & 0 & 0 & 0 &0  & 0 & -1 & x & 0 & u_3 \\
       0 & 0 & 0 & 0 & 0 & 0 & 0 & -1 & x & u_4 \\
       x & 0 & 0 & 0 & 0 & 0 & 0 & 0 & -1 & x+u_5 \\
      \end{array}
    \right)$$\normalsize

\noindent(4) For $s=k+1,\dots,n-1$, we  add   $\frac{1}{x}R(s)$ to row $R(s+1)$:
\tiny
$$\left(
      \begin{array}{cccccccccc}
        x+1 & -(x+1)/x & -(x+1)/x^2 & -(x+1)/x^3 & a_{1,5} & a_{1,6} & a_{1,7} & 0 & 0 & 1 \\
        0 & x & -1/x & -1/x^2 & a_{2,5} & a_{2,6} & a_{2,7} & 0 & 1 & 0 \\
        0 & 0 & x & -1/x & a_{3,5} & a_{3,6} & a_{3,7} & 1 & 0 & 0 \\
        0 & 0 & 0 & x & a_{4,5} & a_{4,6} & a_{4,7} & 0 & 0 & 0 \\
        0 & 0 & 0 & 0 & x & 0 & 0 & 0 & 0 & u_0 \\
        0 & 0 & 0 & 0 & 0 & x & 0 & 0 & 0 & \sum_{i=0}^1 u_i/x^{1-i}\\
        0 & 0 & 0 & 0 & 0 & 0 & x & 0 & 0 & \sum_{i=0}^2 u_i/x^{2-i} \\
       0 & 0 & 0 & 0 &0  & 0 & 0 & x & 0 & \sum_{i=0}^3 u_i/x^{3-i} \\
       0 & 0 & 0 & 0 & 0 & 0 & 0 & 0 & x & \sum_{i=0}^4 u_i/x^{4-i} \\
       x & 0 & 0 & 0 & 0 & 0 & 0 & 0 & 0 & x+\sum_{i=0}^5 u_i/x^{5-i} \\
      \end{array}
    \right)$$\normalsize

\medskip

Let us denote by $v_s:=\sum_{i=0}^s \frac{u_i}{x^{s-i}}$, $s=0, \dots, n-k-1$.

\noindent (5) For $s=3,\dots, k$, we add $\sum_{i=2}^{s-1} \frac{C(i)}{x^{s-i+1}}$ to column $C(s)$:
\tiny
$$\left(
      \begin{array}{cccccccccc}
        x+1 & -(x+1)/x & -(x+1)(1/x^2+1/x^3) & -(x+1)(1/x^3+2/x^4+1/x^5) & a_{1,5} & a_{1,6} & a_{1,7} & 0 & 0 & 1 \\
        0 & x & 0 & 0 & a_{2,5} & a_{2,6} & a_{2,7} & 0 & 1 & 0 \\
        0 & 0 & x & 0 & a_{3,5} & a_{3,6} & a_{3,7} & 1 & 0 & 0 \\
        0 & 0 & 0 & x & a_{4,5} & a_{4,6} & a_{4,7} & 0 & 0 & 0 \\
        0 & 0 & 0 & 0 & x & 0 & 0 & 0 & 0 & v_0 \\
        0 & 0 & 0 & 0 & 0 & x & 0 & 0 & 0 & v_1\\
        0 & 0 & 0 & 0 & 0 & 0 & x & 0 & 0 & v_2 \\
       0 & 0 & 0 & 0 &0  & 0 & 0 & x & 0 & v_3 \\
       0 & 0 & 0 & 0 & 0 & 0 & 0 & 0 & x & v_4 \\
       x & 0 & 0 & 0 & 0 & 0 & 0 & 0 & 0 & x+v_5 \\
      \end{array}
    \right)$$\normalsize

\medskip

Let us now prove by induction on $k$ that, for any $t\le k$, in the above matrix the coefficient in position $(1,t)$ is $-z_t$, where $$z_t:=(x+1)\left(\sum_{r=0}^{t-2} {\binom{t-2}{r}\over x^{t-1+r}}\right).$$
Indeed, for $k=1$ there is nothing to prove. Let us suppose that $$z_i:=(x+1)\left(\sum_{r=0}^{i-2} {\binom{i-2}{r}\over x^{i-1+r}}\right)$$ for $i=1,\dots,t-1$, then one computes that

\begin{align*}
  z_t&=(x+1)/x^{t-1}+z_2/x^{t-1}+\cdots+z_{t-2}/x^3+ z_{t-1}/x^2\\&=(x+1)\left(1/x^{t-1}+\sum_{i=0}^{t-3}\sum_{j=0}^i {\binom{i}{j}\over x^{t+j}}\right)
  =(x+1)\left(1/x^{t-1}+\sum_{i=0}^{t-3}\sum_{j=i}^{t-3-i} {\binom{j}{i}\over x^{t+i}}\right)\\&=(x+1)\left(1/x^{t-1}+\sum_{i=0}^{t-3} {\sum_{j=i}^{t-3-i} \binom{j}{i}\over x^{t+i}}\right)
  =(x+1)\left(1/x^{t-1}+\sum_{i=0}^{t-3} { \binom{t-2}{i+1}\over x^{t+i}}\right)\\
  &=(x+1)\left(1/x^{t-1}+\sum_{i=1}^{t-2} { \binom{t-2}{i}\over x^{t-1+i}}\right)=(x+1)\left(\sum_{r=0}^{t-2} {\binom{t-2}{r}\over x^{t-1+r}}\right).
\end{align*}

\noindent (6) For $s=k+1,\dots, n-k+1$, we add $-\sum_{i=2}^{s-1} \frac{a_{i,s}}{x} C(i)$ to column $C(s)$:
 \tiny
$$\left(
      \begin{array}{cccccccccc}
        x+1 & -z_2 & -z_3 & -z_4 & a_{1,5}+\sum_{i=2}^k a_{i,5}z_i/x &a_{1,6}+\sum_{i=2}^k a_{i,6}z_i/x & a_{1,7}+\sum_{i=2}^k a_{i,7}z_i/x & 0 & 0 & 1 \\
        0 & x & 0 & 0 & 0 & 0 & 0 & 0 & 1 & 0\\
        0 & 0 & x & 0 & 0 & 0 & 0 & 1 & 0 & 0 \\
        0 & 0 & 0 & x & 0 & 0 & 0 & 0 & 0 & 0\\
        0 & 0 & 0 & 0 & x & 0 & 0 & 0 & 0 & v_0 \\
        0 & 0 & 0 & 0 & 0 & x & 0 & 0 & 0 & v_1\\
        0 & 0 & 0 & 0 & 0 & 0 & x & 0 & 0 & v_2 \\
       0 & 0 & 0 & 0 &0  & 0 & 0 & x & 0 & v_3 \\
       0 & 0 & 0 & 0 & 0 & 0 & 0 & 0 & x & v_4 \\
       x & 0 & 0 & 0 & 0 & 0 & 0 & 0 & 0 & x+v_5 \\
      \end{array}
    \right)$$\normalsize

\medskip
Let us denote by $$w_s:=a_{1,s}+\sum_{i=2}^k a_{i,s}z_i/x,$$ $s=k+1, \dots, n-k+1$.

\noindent (7) For $s=2,\dots,k-1$, we add $-\frac{1}{x}C(s)$ to column $C(n-s+1)$:
 \tiny
$$\left(
      \begin{array}{cccccccccc}
        x+1 & -z_2 & -z_3 & -z_4 & w_5 & w_6 & w_7 & z_3/x & z_2/x & 1 \\
        0 & x & 0 & 0 & 0 & 0 & 0 & 0 & 0 & 0\\
        0 & 0 & x & 0 & 0 & 0 & 0 & 0 & 0 & 0 \\
        0 & 0 & 0 & x & 0 & 0 & 0 & 0 & 0 & 0\\
        0 & 0 & 0 & 0 & x & 0 & 0 & 0 & 0 & v_0 \\
        0 & 0 & 0 & 0 & 0 & x & 0 & 0 & 0 & v_1\\
        0 & 0 & 0 & 0 & 0 & 0 & x & 0 & 0 & v_2 \\
       0 & 0 & 0 & 0 &0  & 0 & 0 & x & 0 & v_3 \\
       0 & 0 & 0 & 0 & 0 & 0 & 0 & 0 & x & v_4 \\
       x & 0 & 0 & 0 & 0 & 0 & 0 & 0 & 0 & x+v_5 \\
      \end{array}
    \right)$$\normalsize

\medskip

Thereby, we have obtained a matrix equivalent to $M$ which is of the form required in Lemma \ref{lemma2}. Therefore,  we get:

\begin{align*}
\vert M\vert&= x^n+v_{n-k-1} x^{n-1}+v_{n-k-1}x^{n-2}+\sum_{s=k+1}^{n-k+1} w_{s} v_{s-k-1}x^{n-2}\\&+\sum_{i=1}^{k-2} {z_{k-i}\over x} v_{n-2k+i}x^{n-2}.
\end{align*}
Next, substituting the $v's$, $w's$ and $z's$ in the above expression, we conclude that
\scriptsize
    \begin{align*}
      &\vert M\vert= x^n+ \sum_{i=0}^{n-k-1} u_i  x^{k+i}+ \sum_{i=0}^{n-k-1} u_i  x^{k+i-1}\\&
      +\sum_{s=k+1}^{n-k+1} \left(a_{1,s}+\sum_{i=2}^k \sum_{r=0}^{i-2} \binom{i-2}{r} a_{i,s}{1\over x^{i+r-2}}+\sum_{i=2}^k \sum_{r=0}^{i-2} \binom{i-2}{r} a_{i,s}{1\over x^{i+r-1}}\right)
      \left(\sum_{j=0}^{s-k-1} u_j/x^{s-k-1-j}\right)x^{n-2}\\&
      +\sum_{i=1}^{k-2} {(1+x)\over x}\left(\sum_{r=0}^{k-i-2} {\binom{k-i-2}{r}\over x^{k-i-1+r}}\right) \left(\sum_{j=0}^{n-2k+i} u_j/x^{n-2k+i-j}\right)x^{n-2}\\&
      =x^n+ \sum_{i=0}^{n-k-1} u_i  x^{k+i}+ \sum_{i=0}^{n-k-1} u_i  x^{k+i-1}\hskip 7.6truecm (1)\\&
       +\sum_{s=k+1}^{n-k+1}\sum_{j=0}^{s-k+1} a_{1,s}u_j x^{n+k+j-s-1}\hskip 8.4truecm (2)\\&
       +\sum_{s=k+1}^{n-k+1} \sum_{i=2}^k \sum_{r=0}^{i-2} \sum_{j=0}^{s-k+1} \binom{i-2}{r} a_{i,s} u_j x^{n+k+j-i-r-s-1}\hskip 5.95truecm (3)\\&
      +\sum_{s=k+1}^{n-k+1} \sum_{i=2}^k \sum_{r=0}^{i-2} \sum_{j=0}^{s-k+1} \binom{i-2}{r} a_{i,s} u_j x^{n+k+j-i-r-s}\hskip 6.25truecm (4)\\&
      +\sum_{i=1}^{k-2}\sum_{r=0}^{k-i-2}\sum_{j=0}^{n-2k+i} \binom{k-i-2}{r} u_j x^{k+j-r-2}\hskip 7.2truecm (5)\\&
      +\sum_{i=1}^{k-2}\sum_{r=0}^{k-i-2}\sum_{j=0}^{n-2k+i} \binom{k-i-2}{r} u_j x^{k+j-r-1}.\hskip 7.2truecm (6)
    \end{align*}
\normalsize

\medskip

From now on, we are going to focus on the elements $a_{i,s}$ and we are going to determine the monomial of smallest degree, among the monomials appearing in the determinant of $M$, where each $a_{i,s}$ appears. Recall that $u_0\ne 0$.

\medskip

\noindent$\bullet$ \underbar{$i=1$}. The element $a_{1,s}$ appears only in the summands labelled by (2)
$$\sum_{s=k+1}^{n-k+1}\sum_{j=0}^{s-k+1} a_{1,s}u_j x^{n+k+j-s-1}.\eqno{(2)}$$
To get the  monomial of smallest degree, the index $j$ has to be equal to $0$:
$$\sum_{s=k+1}^{n-k+1} a_{1,s}u_0 x^{n+k-s-1},$$
hence each $a_{1,s}$ appears in the monomial $a_{1,s}u_0 x^{n+k-s-1}$, $s=k+1,\dots, n-k+1$.

\medskip

Let us we represent the degrees of such monomials in the following table:
\begin{center}
\begin{tabular}{| c | c | c| c| c |}
\hline
\rowcolor[gray]{ 0.9}  degree  & $s=k+1$ & \dots& \dots & $s=n-k+1$ \\
\hline
 \cellcolor{lightgray}$i=1$ &  $n-2$ & \dots& \dots & $2k-2$ \\
\hline
\end{tabular}
\end{center}

\medskip

\noindent$\bullet$ \underbar{$i>1$}. Set $i>1$ and $s\in\{k+1,\dots, n-k+1\}$ and let us determine the monomial of smallest degree, where $a_{i,s}$ appears: it only appears in summands labelled by (3) and (4) in the expansion of the determinant of $M$ and, consequently, it is in the summand labelled by $(3)$ where it appears with the smallest degree (notice that the summand labelled by $(4)$ coincides with the summand $(3)$ multiplied by $x$):
$$
\sum_{s=k+1}^{n-k+1} \sum_{i=2}^k \sum_{r=0}^{i-2} \sum_{j=0}^{s-k+1} \binom{i-2}{r} a_{i,s} u_j x^{n+k+j-i-r-s-1}.\eqno{(3)}
$$
There, the index $j$ has to be as small as possible (whence $j=0$) and $r$ has to be as large as possible (so $r=i-2$):
$$\sum_{s=k+1}^{n-k+1} \sum_{i=2}^k  \binom{i-2}{i-2} a_{i,s} u_0 x^{n+k-2i-s+1},$$
hence each $a_{i,s}$ appears in the monomial $ u_0 x^{n+k-2i-s+1}$. Representing the degrees of such monomials in a table we get
\begin{center}
\begin{tabular}{| c | c | c| c| c |}
\hline
\rowcolor[gray]{ 0.9} degree  & $s=k+1$ & \dots& \dots & $s=n-k+1$ \\
\hline
 \cellcolor{lightgray}$i>1$ &  $n-2i$ &\dots & \dots & $2k-2i$ \\
\hline
\end{tabular}
\end{center}

Let us now collect all the information about the degrees in a single table:
{\scriptsize
\begin{center}
\begin{tabular}{| c | c | c| c|c| c| c |}
\hline
\rowcolor[gray]{ 0.9} degree  & $s=k+1$ & $s=k+2$&  \dots& $s=n-k-1$ & $s=n-k$& $s=n-k+1$ \\
\hline
 \cellcolor{lightgray}$i=1$ &  $n-2$ & $n-3$&\dots & $2k$ & $2k-1$ & $2k-2$ \\
\hline
\cellcolor{lightgray}$i=2$ &  $n-4$ & $n-5$&\dots  & $2k-2$ & $2k-3$ & $2k-4$ \\
\hline
\cellcolor{lightgray}$i=3$ &  $n-6$ & $n-7$ &\dots & $2k-4$ & $2k-5$ & $2k-6$ \\
\hline
\cellcolor{lightgray}$\vdots$ &  $\vdots$ &  $\vdots$ &  &$\vdots$ & $\vdots$& $\vdots$ \\
\hline
\cellcolor{lightgray}$i=k-2$ &  $n-2k+4$ & $n-2k+3$ & \dots&  6& 5 & 4\\
\hline
\cellcolor{lightgray}$i=k-1$ &  $n-2k+2$ & $n-2k+1$& \dots & 4 & 3 & 2 \\
\hline
\cellcolor{lightgray}$i=k$ &  $n-2k$ &$n-2k-1$ &\dots & 2 & 1 & 0 \\
\hline
\end{tabular}
\end{center}
}

If we concentrate on the last row and the first two columns of this table, we obtain all the possible degrees from 0 to $n-2$.

For $i=1,\dots, k-1$ and $s=k+1,k+2$, as well as for $i=k$ and $s=k+1,\dots, n-k+1$, let us define the polynomial $f_{i,s}$ consisting of the sum of all the monomials of the determinant of $M$ which contain the variable $a_{i,s}$. Thus, $\{f_{i,s}\}$ is a family of $n-2$ polynomials of smallest degree from 0 to $n-2$, which implies that the set $\{f_{i,s}\}$ forms a basis of the vector space of polynomials of degree less than or equal to $n-2$ over $\mathbb{F}$.

Therefore, for any monic polynomial $q(x)$ of degree $n$ whose trace coincides with the trace of $p(x)$, we can choose appropriate elements $a_{i,s}$ and use such basis to get $q(x)$ as characteristic polynomial of $A+N$,  which completes the proof.
\end{proof}

We will show now that the restriction $k<n-k$ in the previous theorem is unremovable. Before giving a counter-example when $n=2k$, let us study how square-zero matrices $N\in \mathbb{M}_{2k}(\mathbb{F})$ must be in order to get an invertible matrix $A+N$
for
$$
A=\left(
    \begin{array}{c|c}
      {\bf 0}_{k,k} & {\bf 0}_{k,k} \\
      \hline
     {\bf 0}_{k,k} & C \\
    \end{array}
  \right)\in \mathbb{M}_{2k}(\mathbb{F}),
$$
where $C$ is invertible and non-derogatory. %

The following assertion is pivotal for our purpose mentioned above.

\begin{proposition}\label{square-zero}
For any matrix of the form
$$
A=\left(
    \begin{array}{c|c}
      {\bf 0}_{k,k} & {\bf 0}_{k,k} \\
      \hline
     {\bf 0}_{k,k} & C(p(x)) \\
    \end{array}
  \right)\in \mathbb{M}_{2k}(\mathbb{F}),
$$
where $C(p(x))$ is the companion matrix of a degree $k$ polynomial whose zero-coefficient is non-zero and any square-zero matrix $N\in \mathbb{M}_{2k}(\mathbb{F})$ such that $A+N$ is invertible, there exists a basis of $\mathbb{F}^n$ with respect to which the matrix $A$ remains the same and $N$ has the form
$$
N=\left(
    \begin{array}{c|c}
      X & -X^2 \\
      \hline
     {\rm Id}_k &  -X \\
    \end{array}
  \right)
$$
for some $X\in \mathbb{M}_k(\mathbb{F})$.
\end{proposition}

\begin{proof}
Suppose that
$$
C(p(x))=\left(
    \begin{array}{ccc c}
      0 & \dots &  0& -u_0 \\
      1 &  & 0 & \vdots \\
       & \ddots &  & -u_{n-2} \\
      0 &  & 1 & -u_{n-1} \\
    \end{array}
  \right)\in \mathbb{M}_k(\mathbb{F})
$$
with $u_0\ne 0$ and
$$
A=\left(
    \begin{array}{c|c}
      {\bf 0}_{k,k} & {\bf 0}_{k,k} \\
      \hline
     {\bf 0}_{k,k} & C(p(x)) \\
    \end{array}
  \right)\in \mathbb{M}_{2k}(\mathbb{F}).
$$
Let $N=(n_{i,j})$ be any square-zero matrix such that $A+N$ is invertible. Then, the matrix $(A+N)^2=A^2+AN+NA$ is also invertible and has the form
$$
(A+N)^2=\left(
    \begin{array}{c|c}
      {\bf 0}_{k,k} & * \\
      \hline
     M & * \\
    \end{array}
  \right),
$$
where $M$ is given by:
{\small
$$
M=\left(
    \begin{array}{ccccc}
      -u_0n_{2k,1} & -u_0n_{2k,2} & \dots & \dots & -u_0n_{2k,k} \\
      n_{k+1,1}-u_1n_{2k,1} & n_{k+1,2}-u_1n_{2k,2} & \dots & \dots &  n_{k+1,k}-u_1n_{2k,k} \\
      n_{k+2,1}-u_2n_{2k,1} & n_{k+2,2}-u_2n_{2k,2} & \dots & \dots& n_{k+2,k}-u_2n_{2k,k}) \\
      \vdots &  \vdots &  &  &  \vdots\\
      n_{2k-1,1}-u_{k-1}n_{2k,1} & n_{2k-1,2}-u_{k-1}n_{2k,2} & \dots & \dots & n_{2k-1,2}-u_{k-1}n_{2k,2} \\
    \end{array}
  \right).
$$
}
Since $u_0\ne 0$, from the form of $M$ we obtain that the block matrix
$$
N_{21}=\left(
         \begin{array}{cccc}
           n_{k+1,1} &  n_{k+1,2} & \dots &  n_{k+1,k} \\
           n_{k+2,1} & n_{k+2,2} & \dots &  n_{k+2,k} \\
           \vdots &   &   & \vdots \\
           n_{2k,1} & n_{2k,2} &  & n_{2k,k} \\
         \end{array}
       \right)\hbox{ in } N=\left(
    \begin{array}{c|c}
      {N}_{11} & N_{12} \\
      \hline
     N_{21} & N_{22} \\
    \end{array}
  \right)
$$
is invertible too.

Let $\{e_1,\dots, e_n\}$ be the canonical basis of $\mathbb{F}^n$, and consider the subspace $S$ of $\mathbb{F}^n$ generated by $e_1,\dots, e_k$; since the block $N_{21}$ is invertible, there exists linear combinations of the first $k$-columns of $N$ that give rise to $e_{k+1}, e_{k+2},\dots, e_{2k}$ modulo $S$, i.e., there exists scalars $\alpha_{ij}\in \mathbb{F}$ such that $$(\alpha_{1i}N^{(1)}+\dots+\alpha_{ki}N^{(k)})+S=e_{k+i}+S,\quad i=1,\dots,k,$$ where $N^{(1)},\dots, N^{(k)}$ denote the first $k$-columns of $N$.
Define the basis $$\{e_1',\dots, e_k', e_{k+1}, \dots, e_{2k}\},$$ where
$$e'_{i}=\alpha_{1i}e_1+\dots+\alpha_{ki}e_k, \quad i=1,\dots, k.
$$
With respect to this basis, the matrix $N$ is of the form
 $$
N=\left(
    \begin{array}{c|c}
     * & * \\
      \hline
     {\rm Id}_k  &*\\
    \end{array}
  \right)
$$ and the matrix $A$ remains as
$$
A=\left(
    \begin{array}{c|c}
      {\bf 0}_{k,k} & {\bf 0}_{k,k} \\
      \hline
     {\bf 0}_{k,k} & C(p(x)) \\
    \end{array}
  \right)\in \mathbb{M}_{2k}(\mathbb{F}).
$$
Suppose that
$$
N=\left(
    \begin{array}{c|c}
     X & Y \\
      \hline
     {\rm Id}_k  &Z\\
    \end{array}
  \right).
$$
From the condition $N^2=0$, we easily infer that $Y=-X^2$ and $Z=-X$, i.e.,
$$
N=\left(
    \begin{array}{c|c}
     X & -X^2 \\
      \hline
     {\rm Id}_k  &-X\\
    \end{array}
  \right)
$$
for some  $X\in \mathbb{M}_k(\mathbb{F})$.
\end{proof}

We are now ready to give the promised counter-example.
\begin{remark}
Let  $\mathbb{F}=\mathbb{R}$, the field of real numbers. Consider the matrix
  $$
A=\left(
    \begin{array}{cc|cc}
      0 &0 & 0 & 0 \\
      0 & 0 & 0 & 0 \\
      \hline
     0 & 0 & 0 &-1 \\
      0 & 0 & 1 & 0 \\
    \end{array}
  \right)\in\mathbb{M}_4(\mathbb{F}).
  $$
We claim that there does not exist any square-zero matrix $N$ such that the characteristic polynomial of $A+N$ is $x^4+1$.
Otherwise, according to Proposition \ref{square-zero}, we can suppose that $N$ is of the form
$$
N=\left(
    \begin{array}{c|c}
     X & -X^2 \\
      \hline
     {\rm Id}_2  &-X\\
    \end{array}
  \right)\hbox{ for some } X=\left(
                    \begin{array}{cc}
                      n_{11} & n_{12} \\
                      n_{21} & n_{22} \\
                    \end{array}
                  \right).
$$
It is readily checked that the characteristic polynomial of $A+N$ is
$$
p_{A+N}(x)=x^4+(n_{12}-n_{21}+1)x^2-(n_{11}+n_{22})x+n_{11} n_{22}-n_{12}n_{21}
$$
and, if this polynomial equals $x^4+1$, then
$$
\left\{
  \begin{array}{ll}
     n_{12}-n_{21}+1=0\\
    n_{11}+n_{22}=0 \\
    n_{11} n_{22}-n_{12}n_{21}=1
  \end{array}
\right.
$$
so that $n_{22}=-n_{11}$ and $n_{21}=n_{12}+1$ would imply $1=-n_{11}^2-n_{12}(n_{12}+1)=-n_{11}^2-n_{12}^2-n_{12}$; therefore, $$n_{11}^2=-n_{12}^2-n_{12}-1,$$ but, for any real number, the left hand side of this equality is positive, while the right hand side of this equality is less than or equal to $-3/4$, which is an absurd.
\end{remark}

Besides, one also observes that:

\begin{remark} In Remark \ref{k0}, we showed how to modify an invertible non-derogatory matrix $A$ by adding a square-zero matrix $N$ to obtain a matrix $A+N$ with any prescribed characteristic polynomial with the same trace as $A$. In addition, by construction, the resulting matrix $A+N$ was again non-derogatory. With that idea in mind, we close this section by posing the following  problem:

\medskip

\noindent{\bf Question.} Given a matrix of the form $$
A=\left(
         \begin{array}{c|c}
           {\bf 0}_{k,k} & {\bf 0}_{k,n-k} \\
           \hline
           {\bf 0}_{n-k,k} & A_{22} \\
         \end{array}
       \right)\in \mathbb{M}_{n}(\mathbb{F})
$$
consisting of $k$ rows and columns of zeros and an invertible non-derogatory matrix $A_{22}$, and a prescribed polynomial $q(x)$ of degree $n$ whose trace coincides with the trace of $A$, is it possible to find a square-zero matrix such that $A+N$ is {\it non-derogatory} and the characteristic polynomial of $A+N$ coincides with $q(x)$?
\end{remark}

\section{Some consequences}

In the present section, some results related to our project of decomposing a square matrix $A$ into the sum of a square-zero matrix $N$ and another matrix satisfying some fundamental properties such as been diagonalizable, invertible, torsion, etc.  (see, for example, \cite{DGL1}, \cite{DGL2}, \cite{DGL3} and \cite{DGL4}) can be obtained by forcing that $A-N$ verifies a certain prescribed characteristic polynomial. Specifically, the key point in \cite{DGL3} when showing that any nilpotent matrix of rank at least $\frac{n}{2}$ can be expressed as the sum of a torsion matrix and a square-zero matrix was the fact that the original nilpotent matrix could be written as the direct sum of blocks, each of them expressed as the sum of a square-zero matrix and a matrix satisfying a certain characteristic polynomial that forced this last matrix to be torsion.

\medskip

So, with the aid of Theorem \ref{teoremappal}, we can elementarily prove the following three corollaries.

\begin{corollary}
Let $\mathbb{F}$ be a field, let $n,k\in \mathbb{N}$ with $k<n-k$, and consider the block matrix
$$
A=\left(
         \begin{array}{c|c}
           {\bf 0}_{k,k} & {\bf 0}_{k,n-k} \\
           \hline
           {\bf 0}_{n-k,k} & A_{22} \\
         \end{array}
       \right)\in \mathbb{M}_{n}(\mathbb{F})
$$
consisting of $k$ rows and columns of zeros and an invertible non-derogatory matrix $A_{22}$. Then, if the field $\mathbb{F}$ has enough elements, there exists a square-zero matrix $N$ such that $A+N$ is diagonalizable (with $n$ different eigenvalues).
\end{corollary}

\begin{proof}
If $\mathbb{F}$ has enough elements, we can take different elements $\alpha_1,\dots, \alpha_n\in \mathbb{F}$  such that $-(\alpha_1+\dots +\alpha_n)$ equals the trace of $A$. Then fix the polynomial $q(x)=(x-\alpha_1)\cdots(x-\alpha_n)$ and use Theorem \ref{teoremappal} to find a square-zero matrix $N$ such that the characteristic polynomial of $A+N$ is $q(x)$. Since $A+N$ has $n$ different eigenvalues, $A+N$ is diagonalizable, as asserted.
\end{proof}

\begin{corollary}
Let $\mathbb{F}$ be a field, let $n,k\in \mathbb{N}$ with $k<n-k$, and consider the block matrix
$$
A=\left(
         \begin{array}{c|c}
           {\bf 0}_{k,k} & {\bf 0}_{k,n-k} \\
           \hline
           {\bf 0}_{n-k,k} & A_{22} \\
         \end{array}
       \right)\in \mathbb{M}_{n}(\mathbb{F})
$$
consisting of $k$ rows and columns of zeros and an invertible matrix $A_{22}$. Then, there exists an invertible matrix $T$ and a square-zero matrix $N$ such that $A=T+N$.
\end{corollary}

\begin{proof}
Without loss of generality, assume that $A_{22}$ is the direct sum of the invertible and non-derogatory companion matrices of its elementary divisors. Combine the $0$'s in the main diagonal of $A$ with each of these invertible and non-derogatory companion matrices and utilize Theorem \ref{teoremappal} to get characteristic polynomials with nonzero zero-degree coefficients by adding appropriate square-zero matrices, as desired.
\end{proof}

Recall that a matrix $P\in \mathbb{M}_{n}(\mathbb{F})$ is {\it $m$-potent} if $P^m=P$ for some $1<m\in\mathbb{N}$; a matrix $T\in \mathbb{M}_{n}(\mathbb{F})$ is {\it $k$-torsion} (or just a {\it $k$-root of the unity}) if $T^k=\Id$ for some $k\in\mathbb{N}$.

\begin{corollary}\label{consequences}
Let $\mathbb{F}$ be a field, let $n,k\in \mathbb{N}$ with $k<n-k$, and consider the block matrix
$$
A=\left(
         \begin{array}{c|c}
           {\bf 0}_{k,k} & {\bf 0}_{k,n-k} \\
           \hline
           {\bf 0}_{n-k,k} & A_{22} \\
         \end{array}
       \right)\in \mathbb{M}_{n}(\mathbb{F})
$$
consisting of $k$ rows and columns of zeros and an invertible non-derogatory matrix $A_{22}$.
If the trace of $A$ is zero, then the following two statements are true:
\begin{itemize}
\item[(i)] if $n\ge 3$, there exists an $n$-potent matrix $P$ and a square-zero matrix $N_1$ such that $A=P+N_1$;
\item[(ii)] if $n\ge 2$, there exists an $n$-torsion matrix $T$ and a square-zero matrix $N_2$ such that $A=T+N_2$;
 \end{itemize}
\end{corollary}

\begin{proof} One observes that it suffices to prescribe the zero-trace polynomials $p_1(x)=x^n-x$ and $p_2(x)=x^n-1$ and employ Theorem \ref{teoremappal} to get square-zero matrices $N_1,N_2$ such that $A+N_i$ has characteristic polynomial equal to $p_i(x)$, $i=1,2$. Then, decompose $A=(A+N_1)-N_1$ to get (i) and $A=(A+N_2)-N_2$ to get (ii), as wanted.
\end{proof}

\medskip
\medskip
\medskip

\noindent{\bf Funding:} The first-named author (Peter Danchev) was supported in part by the BIDEB 2221 of T\"UB\'ITAK; the second and third-named authors (Esther Garc\'{\i}a and Miguel G\'omez Lozano) were partially supported by Ayuda Puente 2023, URJC, and MTM2017-84194-P (AEI/FEDER, UE). The all three authors were partially supported by the Junta de Andaluc\'{\i}a FQM264.

\vskip3.0pc

\bibliographystyle{plain}

\end{document}